\documentclass[11pt,a4paper]{amsart}

\usepackage{amscd,amsthm,amsfonts,amssymb,amsmath,comment,supertabular}
\usepackage{mathrsfs}
\usepackage[colorlinks=true]{hyperref}
\usepackage{fullpage}
\usepackage[all]{xy}
\usepackage{xcolor}

    \newcommand{\BA}{{\mathbb {A}}} 
    \newcommand{\BC}{{\mathbb {C}}} 
     \newcommand{\BF}{{\mathbb {F}}}

     \newcommand{\BP}{{\mathbb {P}}}
    \newcommand{\BQ}{{\mathbb {Q}}} \newcommand{\BR}{{\mathbb {R}}}

     \newcommand{\BZ}{{\mathbb {Z}}}

     \newcommand{\CH}{{\mathcal {H}}}

    \newcommand{\CO}{{\mathcal {O}}}

     \newcommand{\RH}{{\mathrm {H}}}

    \newcommand{\ab}{{\mathrm{ab}}}
    
    \newcommand{\alg}{{\mathrm{alg}}}
    
     \newcommand{\an}{{\mathrm{an}}}

    \newcommand{\Gal}{{\mathrm{Gal}}} \newcommand{\GL}{{\mathrm{GL}}}

    \renewcommand{\Im}{{\mathrm{Im}}}

    \newcommand{\Jac}{{\mathrm{Jac}}}\newcommand{\Ker}{{\mathrm{Ker}}}

    \newcommand{\ord}{{\mathrm{ord}}} \newcommand{\rk}{{\mathrm{rank}}}

    \renewcommand{\mod}{\ \mathrm{mod}\ }

    \newcommand{\Sel}{{\mathrm{Sel}}}

    \newcommand{\sgn}{{\mathrm{sgn}}}

    \newcommand{\tr}{{\mathrm{tr}}}\newcommand{\tor}{{\mathrm{tor}}}
    
    \newcommand{\ur}{{\mathrm{ur}}}

        \newcommand{\sN}{\mathscr{N}}

\DeclareFontFamily{U}{wncy}{}
\DeclareFontShape{U}{wncy}{m}{n}{<->wncyr10}{}
\DeclareSymbolFont{mcy}{U}{wncy}{m}{n}
\DeclareMathSymbol{\Sha}{\mathord}{mcy}{"58}

    \newcommand{\wt}{\widetilde}
    \newcommand{\wh}{\widehat}

    \newcommand{\ov}{\overline}

    \newcommand{\ra}{\rightarrow}

\newcommand{\Cor}[1]{}

    \theoremstyle{plain}
    \newtheorem{thm}{Theorem}[section] \newtheorem{coro}[thm]{Corollary}
    \newtheorem{lem}[thm]{Lemma}  \newtheorem{prop}[thm]{Proposition}
     \newtheorem{defn}[thm]{Definition}

\theoremstyle{remark} \newtheorem{remark}{Remark}[section]
\theoremstyle{remark} 
\theoremstyle{remark} 

    \numberwithin{equation}{section}

\begin{document}
\title[Generalized Birch's lemma]{Generalized Birch lemma and the 2-part of the Birch and Swinnerton-Dyer conjecture for certain elliptic curves}

\author[J. Shu]{Jie Shu}
\address{School of Mathematical Sciences, Tongji University, Shanghai 200092,  P. R. China}
\email{shujie@tongji.edu.cn}

\author[S. Zhai]{Shuai Zhai}
\address{Department of Pure Mathematics and Mathematical Statistics, University of Cambridge, Cambridge CB3 0WB, UK.}
\email{S.Zhai@dpmms.cam.ac.uk}

\thanks{Jie Shu is supported by NSFC-11701092.}

\begin{abstract}
In the present paper, we generalize the celebrated classical lemma of Birch and Heegner on quadratic twists of elliptic curves over $\BQ$. We prove the existence of explicit infinite families of quadratic twists with analytic ranks $0$ and $1$ for a large class of elliptic curves, and use Heegner points to explicitly construct rational points of infinite order on the twists of rank $1$. In addition, we show that these families of quadratic twists satisfy the $2$-part of the Birch and Swinnerton-Dyer conjecture when the original curve does. We also prove a new result in the direction of the Goldfeld conjecture.
\end{abstract}


\maketitle

\tableofcontents

\section{Introduction}
Let $E$ be an elliptic curve defined over $\BQ$, and let $N$ be the  conductor of $E$. Then $E$ is modular by the theorem of Wiles et al \cite{Wiles95} \cite{BCDT01}, and we let $f: X_0(N) \ra E$ be an optimal modular parametrization sending the cusp at infinity, which is denoted by $[\infty]$, to the zero element of $E$. We write $[0]$ for the cusp of $X_0(N)$ arising from the zero point in $\BP^1(\BQ)$ under the complex uniformization of $X_0(N)$.  By the theorem of Manin--Drinfeld, $f([0])$ is a torsion point in $E(\BQ)$.  For any square-free integer $D \neq 1$, we write $E^{(D)}$ for the twist of $E$ by the extension $\BQ(\sqrt{D})/\BQ$, and $L(E^{(D)}, s)$ for its complex $L$-series. We assume throughout this paper that the group $E[2](\BQ)$ of rational 2-division points on $E$ is cyclic of order 2, and we define the elliptic curve $E'/\BQ$ to be the quotient curve
$E' = E/E[2](\BQ)$. In all that follows, we will make the following assumption:
\[E[2](\BQ)=E'[2](\BQ)=\BZ/2\BZ.\leqno{\mathrm{(\bf{Tor})}}\]
Thus  $\BQ(E[2])$ and $\BQ(E'[2])$ are quadratic extensions of $\BQ$.  We also remark that it is proven in Proposition \ref {Hecke4} that, for any elliptic curve $E$ over $\BQ$ with $f([0])\not\in 2E(\BQ)$,  condition $(\mathrm{Tor})$ holds for $E$ if and only if there exists an odd prime $q$ of good reduction for $E$ such that 
\begin{equation}\label{hecke}
a_{q}\equiv 1-\left(\frac{-1}{q}\right)\mod 4,
\end{equation} 
where  $a_q$ is the trace of Frobenius on the reduction of $E$ modulo $q$.
We will use the following explicit set of primes, which, by Chebotarev's theorem, has positive density in the set of all primes.
\begin{defn}\label{def1}
A prime $q$ is \textbf{admissible} for $E$ if $(q,2N)=1$ and $q$ is inert in both of the quadratic fields $\BQ(E[2])$ and $\BQ(E'[2])$.
\end{defn}

\noindent Note also that, assuming that our elliptic curve $E$ satisfies condition $(\mathrm{Tor})$, we prove in Proposition \ref{Hecke3} that an odd prime $q$ of good reduction is admissible for $E$ in the above sense if and only if it satisfies \eqref{hecke}.

In what follows, $p$ will always denote a prime  $>3$ such that $p \equiv 3 \mod 4$, and we will write  $K=\BQ(\sqrt{-p})$.  We say that $p$ satisfies \textbf{\emph{Heegner hypothesis}} for $(E,K)$ if every prime $\ell$ dividing $N$ splits in the field $K$. For any odd prime $q$, put $q^*=\left(\frac{-1}{q}\right)q$, where $\left({\frac {\cdot}{\cdot}}\right)$ denotes the Legendre symbol. We shall prove the following result, which generalizes an old lemma of Birch \cite{Birch70}.
\begin{thm}\label{main1}
Let $E$ be an elliptic curve over $\BQ$ satisfying $f([0])\not\in 2E(\BQ)$ and Condition $(\mathrm{Tor})$. Let $p>3$ be any prime $\equiv 3 \mod 4$ satisfying the Heegner hypothesis for $(E,K)$, where $K=\BQ(\sqrt{-p})$. For any integer $r\geq 0$, let $q_1, \ldots, q_r$ be distinct admissible primes which are not equal to  $p$, and put $M=q_1^* q_2^* \cdots q_r^*$. Then we have
$$
\ord_{s=1}L(E^{(M)}, s)=\rk\  E^{(M)}(\BQ)=0,  \text{ \ and \ } 
\ord_{s=1}L(E^{(-pM)}, s)=\rk\  E^{(-pM)}(\BQ)=1.
$$
Moreover, the Shafarevich--Tate groups $\Sha(E^{(M)})$ and $\Sha(E^{(-pM)})$ are both finite.
\end{thm}

In view of the above remarks, we immediately obtain the following corollary.
\begin{coro}\label{cor1}
Let $E$ be any elliptic curve over $\BQ$ such that $f([0])\not\in 2E(\BQ)$, and there exists an odd prime $q$ of good reduction for $E$ satisfying \eqref{hecke}.
Then, for any integer $r\geq 1$, there exist infinitely many square-free integers $M$ and $M'$, having exactly $r$ prime factors, such that $L(E^{(M)}, s)$ does not vanish at $s=1$, and $L(E^{(M')}, s)$ has a zero of order $1$ at $s=1$.
\end{coro}

\noindent Theorem \ref{main1} and Corollary \ref{cor1} improve the results in \cite[Theorem 1.1]{CLTZ15} and \cite[Theorem 1.1]{CLW16}, which were motivated by the remarkable progress on the congruent number problem made by Tian \cite{Tian14}. Note, however, that, in these two papers, much stronger  conditions are imposed on the prime factors of $M$ and $M'$. In particular, the generalizations of the Birch lemma in \cite[Corollary 2.6]{CLTZ15} and \cite[Theorem 1.1]{CLW16} require that the prime factors $q$ of $M$ and $M'$ satisfy  $q\equiv 1\mod 4$ and $a_q \equiv 0 \mod 2^{r+1}$, where, as above, $r$ denotes the number of prime factors of $M$ or $M'$. Theorem \ref{main1} and Corollary \ref{cor1}, can be applied to many elliptic curves, especially to elliptic curves without complex multiplication. We present a wide range of examples in Section \ref{eg}.  Note also that in Theorem \ref{main1}, $M$ will be negative precisely when the number of admissible primes $q\equiv 3\mod 4$ dividing $M$ is odd, whence $-pM$ will be positive. In some cases, this enables us to use Heegner points to construct rational points of infinite order on real quadratic twists of our elliptic curve $E$, as for example in Section 5.1, where $E$ is a Neumann--Setzer curve. 

\medskip

For the $2$-part of the Birch and Swinnerton-Dyer conjecture for the quadratic twists of $E$ occurring in Theorem \ref{main1}, we prove the following result (see Theorem \ref{2BSD}). 

\begin{thm}\label{main2}
Let $E$ and $M$ be as in Theorem \ref{main1}. Assume that  (i) the Manin constant of $E$ is odd, and (ii) every prime $\ell$ dividing $2N$ splits in both $K=\BQ(\sqrt{-p})$ and $\BQ(\sqrt{M})$. Then, if the $2$-part of the Birch and Swinnerton-Dyer conjecture holds for $E$, it also holds for both $E^{(M)}$ and $E^{(-pM)}$.
\end{thm}
In Section 5.3, we give infinitely many examples of elliptic curves over $\BQ$, without complex multiplication, which satisfy the full Birch and Swinnerton-Dyer conjecture, by combining this result with a recent theorem of Wan \cite{Wannt} on the $p$-part of the Birch and Swinnerton-Dyer conjecture for primes $p > 2$, whose proof uses deep methods from Iwasawa theory.

We now present another application of Theorem \ref{main1}. In 1979, Goldfeld \cite{Goldfeld79} conjectured that, for every elliptic curve defined over $\BQ$, amongst the set of all its quadratic twists with root number $+1$, there is a subset of density one where the central $L$-value of the twist is non-zero, and amongst the set of all its quadratic twists with root number $-1$, there is a subset of density one where the central $L$-value of the twist has a zero of order equal to 1. For an elliptic curve $E$ over $\BQ$, define
$$
N_{r}(E,X) = \#\{\text{square free }D \in \BZ: \lvert D \rvert \leq X, \ \ord_{s=1} L(E^{(D)},s)=r\},
$$ 
where $r=0, 1$. We prove the following result.
\begin{thm}\label{main3}
Let $E$ be an elliptic curve over $\BQ$ satisfying $f([0])\not\in 2E(\BQ)$, and Condition $(\mathrm{Tor})$. Then, as $X \to \infty$, we have 
$$
N_{r}(E,X) \gg \frac{X}{\log^{3/4}X}.
$$
\end{thm}
\noindent Previously, for any given elliptic curve $E$ over $\BQ$, Ono and Skinner \cite{OS98} showed that 
$$
N_{0}(E,X)\gg \frac{X}{\log X},
$$
and later Ono \cite{Ono01} proved that, in particular for $E$ without a rational $2$-torsion point, 
$$
N_{0}(E,X)\gg \frac{X}{\left(\log X\right)^{1-\alpha}}
$$ 
for some $0<\alpha<1$. Perelli and Pomykala \cite{PP97} proved that $N_{1}(E,X)\gg X^{1-\varepsilon}$ for every $\varepsilon>0$. Much progress for various families of elliptic curves has been made towards Goldfeld's conjecture. For example, see the early survey article by Silverberg \cite{Silverberg07}, and very recent results in \cite{Smithnt} and \cite{KL19}. However, Theorem \ref{main3} improves previous results when $E$ satisfies $f([0])\not\in 2E(\BQ)$ and Condition $(\mathrm{Tor})$. 
The proof of Theorem \ref{main3} is a straightforward consequence of Theorem \ref{main1}. Indeed, we just count the number of integers $M$ appearing in Theorem \ref{main1}, say
$$
S_0(X) :=\sum_{|M| \leq X}1,
$$ 
where $M = q_1^* \cdots q_r^*$ as in Theorem \ref{main1}. In view of Proposition \ref{Hecke3}, the Chebotarev density theorem shows that the admissible primes have natural density $1/4$ in the set of all primes. Applying the Ikehara Tauberian theorem, it follows that 
$$
S_0(X) \sim c_0 X/ (\log X)^{1-1/4} \sim c_0 X /(\log X)^{3/4},
$$
where $c_0$ is a constant. Then Theorem \ref{main3} follows immediately since $N_{r}(E,X) \gg S_0(X)$.

\bigskip

\section{The arithmetic of the elliptic curve $E$}
 Throughout the rest of the paper, we shall always assume that  $E$ is an elliptic curve over $\BQ$ satisfying $f([0])\not\in 2E(\BQ)$ and Condition $(\mathrm{Tor})$. In this section, we establish some of the basic arithmetic properties of these curves.

\subsection{Condition $(\mathrm{Tor})$ and admissible primes}
Under Condition $\mathrm{(Tor)}$, the quotient map $\phi:E\ra E'$ is an isogeny over $\BQ$ of degree $2$. Let $\phi': E' \to E$ be its dual isogeny.
\begin{prop}\label{tor}
Under Condition $\mathrm{(Tor)}$, we have $E[2^\infty](\BQ)=E[2](\BQ)=\BZ/2\BZ$. 
\end{prop}
\begin{proof}
Suppose on the contrary that there is a point $P_1$ in $E(\BQ)$ of exact order 4, it follows that the image of $P_1$ in $E'$, namely $\phi(P_1)$, is rational by the definition of $2$-isogeny. Since $2P_1$ generates $E[2](\BQ)$, we have $\phi(2P_1)=0$, but $P_1$ does not lie in the kernel of $\phi$, so $\phi(P_1)$ is of order $2$, and $\phi' \circ \phi(P_1)=2P_1$. Let  $P_2\in E[2]\backslash E[2](\BQ)$ be a point of order $2$. Then $\phi(P_2)$ must be rational of order $2$ since $\phi' \circ \phi(P_2)=0$. Thus, the images of $P_1$ and $P_2$ in $E'=E/E[2](\BQ)$ are different rational points of order $2$ and they generate $E'[2]$, that is, $E'[2](\BQ)=E'[2]=(\BZ/2\BZ)^2$, which does not satisfy Condition $\mathrm{(Tor)}$.
\end{proof}

Note $\BQ(E[2])=\BQ(\sqrt{\Delta_E})$ and $\BQ(E'[2])=\BQ(\sqrt{\Delta_{E'}})$ are quadratic extensions of $\BQ$, where $\Delta_E$ and $\Delta_{E'}$ are the discriminants of $E$ and $E'$, respectively.  For any prime $q$ of good reduction for $E$, let $a_q$ be the trace of the Frobenius on the reduction of $E$ mod  $q$.
\begin{prop}\label{Hecke3}
Let $q$ be a prime with $(q,2N)=1$. Under Condition $\mathrm{(Tor)}$, we have that
$$
a_q
\equiv
\left\{
\begin{array}{llll}
0 \ \mod 4 & \hbox{if $q \equiv 1 \mod 4$ and $q$ is inert in both $\BQ(\sqrt{\Delta_E})$ and $\BQ(\sqrt{\Delta_{E'}})$;} \\
0 \ \mod 4 & \hbox{if $q \equiv 3 \mod 4$ and $q$ splits in $\BQ(\sqrt{\Delta_E})$ or $\BQ(\sqrt{\Delta_{E'}})$;} \\
2 \ \mod 4 & \hbox{if $q \equiv 1 \mod 4$ and $q$ splits in $\BQ(\sqrt{\Delta_E})$ or $\BQ(\sqrt{\Delta_{E'}})$;} \\
2 \ \mod 4 & \hbox{if $q \equiv 3 \mod 4$ and $q$ is inert in both $\BQ(\sqrt{\Delta_E})$ and $\BQ(\sqrt{\Delta_{E'}})$.}
\end{array}
\right.
$$
In particular, $q$ is inert in both $\BQ(\sqrt{\Delta_E})$ and $\BQ(\sqrt{\Delta_{E'}})$, i.e., admissible for $E$ if and only if  

\begin{equation}\label{Hecke}
a_q\equiv 1- \left(\frac{-1}{q}\right) \mod 4.
\end{equation} 

\end{prop}
\begin{proof}
The elliptic curve $E$ has good reduction at $q$. Now reduction modulo $q$, gives the exact sequence
\[0\ra \wh{E}(q\BZ_q)\ra E(\BQ_q)\ra E(\BF_q)\ra 0,\]
where $\wh{E}/\BZ_q$ denotes the formal group associated to $E/\BQ_q$.
Since $q>2$, by \cite[Chapter IV, Theorem 6.4 ]{Silvermanbook1}, multiplication by $4$ induces an isomorphism on $\wh{E}(q\BZ_q)$, and hence we have $E[4](\BQ_q)=E[4](\BF_q)$. If $q$ splits in $\BQ(\sqrt{\Delta_E})$ or $\BQ(\sqrt{\Delta_{E'}})$, then either 
$$
E[2](\BF_q)=E[2](\BQ_q)=(\BZ/2\BZ)^2
$$
or 
$$
E'[2](\BF_q)[2]=E'[2](\BQ_q)[2]=(\BZ/2\BZ)^2.
$$
Since $E$ is isogenous to $E'$, $|E(\BF_q)|=|E'(\BF_q)|$. In either case, $4$ divides $|E(\BF_q)|$. If $q$ is inert in both $\BQ(\sqrt{\Delta_E})$ and $\BQ(\sqrt{\Delta_{E'}})$, by \cite[Lemma 2.1]{Zhaint}, we have 
$$
E[4](\BF_q)=E[4](\BQ_q)=\BZ/2\BZ.
$$
In conclusion, we have 
$$
|E(\BF_q)| \equiv
\left\{
\begin{array}{ll}
2 \ \mod 4 & \hbox{if $q$ is inert in both $\BQ(\sqrt{\Delta_E})$ and $\BQ(\sqrt{\Delta_{E'}})$;} \\
0 \ \mod 4 & \hbox{if $q$ splits in $\BQ(\sqrt{\Delta_E})$ or $\BQ(\sqrt{\Delta_{E'}})$.}
\end{array}
\right.
$$
Since $q$ is an odd good prime for $E$, the assertion follows immediately from the injection $E(\BQ_q)[2^\infty]\hookrightarrow E(\BF_q)$ under reduction modulo $q$ and  the formula $a_q=q+1-|E(\BF_q)|$. 
\end{proof}

\begin{prop}\label{Hecke4}
If $f([0])\notin 2E(\BQ)$, then condition $\mathrm{(Tor)}$ is equivalent to the existence of an odd good prime $q$ with 
\begin{equation*}
a_q\equiv 1- \left(\frac{-1}{q}\right) \mod 4. 
\end{equation*} 
\end{prop}
\begin{proof}
Under Condition $\mathrm{(Tor)}$, by Proposition \ref{Hecke3}, admissible primes are exactly those primes $q\nmid 2N$ satisfying (\ref{Hecke}).

Suppose $q\nmid 2N$ satisfies (\ref{Hecke}). Then
\[|E(\BF_q)|=q+1-a_q\equiv 2\mod 4,\]
which implies that both $|E[2^\infty](\BQ)|$ and $|E'[2^\infty](\BQ)|$ are less than or equal to $2$ for $q\nmid 2N$. Since $f([0])\notin 2E(\BQ)$, we have $E[2^\infty](\BQ)\neq 0$ and hence $E[2^\infty](\BQ)=\BZ/2\BZ$. Then the quotient map $\phi: E\ra E'$ is an isogeny over $\BQ$ of degree $2$, and we write $\phi':E'\ra E$ for its dual isogeny. Then the kernel of $\phi'$ is rational over $\BQ$, and hence contains a rational point of order $2$. Therefore $E'[2^\infty](\BQ)=\BZ/2\BZ$, and hence Condition $\mathrm{(Tor)}$ holds.
\end{proof}

\subsection{Equivalent condition for the $2$-part of the Birch and Swinnerton-Dyer conjecture}
In the rest of this section, we assume that the Manin constant $c_E$ of $E$ is always odd. Let $\Omega^+$ (respectively, $\Omega^-$) be the minimal real (respectively, imaginary) period of $E$. \begin{prop}\label{L1}
For the elliptic curve $E$ we have $L(E,1)\neq 0$, and
\[\ord_2\left(\frac{L(E,1)}{\Omega^+}\right)=-1.\]
\end{prop}
\begin{proof}
Let $f:X_0(N)\ra E$ be an optimal modular parametrization.  We embed $X_0(N)$ into its Jacobian $J_0(N)$ via the base point $[\infty]$. Then there exists a unique homomorphism $\wt{f}:J_0(N)\ra E$ through which $f$ factors. We consider the complex uniformization of $\wt{f}$. By the Abel--Jacobi theorem, we have 
\[J_0(N)=\RH^1(X_0(N),\BR)/\RH_1(X_0(N),\BZ)  \text{ \ and \ } \Jac(E)=\RH^1(E,\BR)/\RH_1(E,\BZ).\]
Then $f$ induces a homomorphism
\[J_0(N)\ra \Jac(E),\quad \gamma\mapsto f(\gamma),\quad \gamma\in \RH^1(X_0(N),\BR).\]
The elliptic curve $E$ has a complex uniformization $E=\BC/L$ where $L\subset \BC$ is the period lattice of $E$, and we have an isomorphism
\[\Jac(E)\simeq E,\quad \alpha\mapsto \int_\alpha \omega_E,\]
where $\omega_E$ is the N\'eron differential of $E$.  Then the complex uniformization of $\wt{f}$ is explicitly given as
\begin{equation}\label{complex}
\wt{f}:J_0(N)\ra E,\quad \gamma \mapsto \int_{f(\gamma)}\omega_E.
\end{equation}
In particular, we take $\gamma=[0\ra \infty i]$ to be the path on $X_0(N)$ joining $[0]$ and $[\infty i].$  Let $g$ be the newform associated to $E$. Since $f^*\omega_E=c_E2\pi i g(z)dz$, it follows that 
\[L(E,1)=2\pi i \int_\gamma g(z)dz=c_E^{-1}\int_{f(\gamma)}\omega_E.\]
Then from the complex uniformization (\ref{complex}), we see that
\begin{equation}\label{L-value}
f([0])=c_EL(E,1)\mod L.
\end{equation}
Let $\RH_1(E,\BZ)^+$ be the submodule of $\RH_1(E,\BZ)$ that is invariant under the complex conjugation. Then $\RH_1(E,\BZ)^+$ is free of rank $1$ and
\[\Omega^+=\left|\int_{\gamma^+} \omega_E\right|,\]
where $\gamma^+$ is a generator of the submodule $\RH_1(E,\BZ)^+$. Since $\gamma$ is a real path on $X_0(N)$, it follows that $f(\gamma)\in \RH_1(E,\BR)^+=\RH_1(E,\BZ)^+\otimes_\BZ\BR$. Let $C$ be a sufficiently large odd integer such that $Cf([0])$ has order $2$. Then by (\ref{L-value}) we have
\[Cc_E\frac{L(E,1)}{\Omega^+}\not \in \BZ \text{ \ and \ }  2Cc_E\frac{L(E,1)}{\Omega^+}\in \BZ.\]
Hence
\[\ord_2\left(\frac{L(E,1)}{\Omega^+}\right)=-1,\]
as desired.
\end{proof}
Denote
\[\Omega=\int_{E(\BR)} |\omega_E|\]
the real period of $E$. Then $\Omega=\Omega^+$ or $2\Omega^+$ according to that $E(\BR)$ is connected or not. The period $\Omega$ is the one appearing in the exact formula which is part of the Birch and Swinnerton-Dyer conjecture and we define
\[L^\alg(E,1)=\frac{L(E,1)}{\Omega}.\]

\begin{prop}\label{dis}
The discriminant $\Delta_E<0$. In particular,  we have \[\ord_2(L^\alg(E,1))=-1.\]  
\end{prop}

\begin{proof}
Suppose, on the contrary, that $\Delta_E>0$. Then the period lattice is $\BZ\Omega^++\BZ\Omega^-$. 
Let $q$ be an admissible prime. By Manin's formula \cite[Theorem 3.3]{Manin72}, we have 
\[|E(\BF_q)|L(E,1)=c_E\sum_{k=1}^{q-1}\langle \{0,k/q\}, g\rangle,\]
where $\{\cdot,\cdot\}$ denotes the modular symbol, and $\langle \{0,k/q\}, g\rangle=s_k\Omega^++t_k\Omega^-$, with $s_k,t_k\in \BZ$. Since 
$$
\ov{\langle \{0,k/q\}, g\rangle}=\langle \{0,1-k/q\}, g\rangle,
$$ 
we have
\[|E(\BF_q)|L(E,1)=2\Omega^+c_E\sum_{k=1}^{\frac{q-1}{2}}s_k.\]
Then by Proposition \ref{L1}, we have $\ord_2(|E(\BF_q)|)\geq 2$. Since $q$ is admissible, that is 
\[a_q=1-\left(\frac{-1}{q}\right)\mod 4,\]
it follows that $|E(\BF_q)|=q+1-a_q$ has $2$-adic valuation $1$, and we get a contradiction.
Thus $\Delta_E<0$, whence $E(\BR)$ is connected and $\Omega=\Omega^+$, and 
\[L^\alg(E,1)=\frac{L(E,1)}{\Omega^+}\]
has $2$-adic valuation $-1$.
\end{proof}

\begin{coro}\label{24E}
We have 
\[\rk E(\BQ)=\ord_{s=1}L(E,s)=0.\]
The $2$-part of the Birch and Swinnerton-Dyer conjecture holds for the elliptic curve $E$ if and only if
\[\ord_2\left(\frac{|\Sha(E)|}{|E(\BQ)_\tor|^2}\prod_\ell c_\ell(E)\right)=-1.\]

\end{coro}
\begin{proof}
By Proposition \ref{L1}, $L(E,1)\neq 0$ and by the work of Gross and Zagier \cite{GZ1986} and Kolyvagin \cite{Kolyvagin1990}, $E(\BQ)$ has rank zero and $|\Sha(E)|$ is finite. The statement for the $2$-part of the Birch and Swinnerton-Dyer conjecture for $E$ follows immediately from Proposition \ref{dis}.
\end{proof}

\subsection{Discriminants and periods}

\begin{prop}\label{Delta}
We have $\Delta_E<0$ and $\Delta_{E'}>0$. 
\end{prop}
\begin{proof}
Moving the rational $2$-torsion point on $E$ to $(0,0)$, we can find a Weierstrass equation for $E/\BQ$ of the form
\begin{equation}\label{eqE}
Y^2 = X^3 + aX^2 + bX,
\end{equation}
where $a,b \in \BZ$. Dividing this curve by the subgroup generated by the point $(0,0)$, we obtain a Weierstrass equation for $E'/\BQ$ of the form
\begin{equation}\label{eqE'}
y^{2}=x^3-2ax^{2}+(a^2-4b)x.
\end{equation}
Note that the discriminants of various Weierstrass equations of elliptic curves over $\BQ$ are the same up to multiplication by non-zero rational twelfth powers. It is easy to calculate that the discriminant of equation \eqref{eqE} is $2^4b^2(a^2-4b)$. It is proved in Proposition \ref{dis} that $\Delta_E<0$, and hence $b>0$. A simple computation shows that the discriminant of equation \eqref{eqE'} is $2^8b(a^2-4b)^2$, which is positive since $b>0$. Therefore, $\Delta_{E'}>0$.
\end{proof}

\begin{lem}\label{period}
Let $\Omega$ and $\Omega'$ denote the least positive real periods of $E$ and $E'$, respectively. Then
\[\frac{\Omega}{\Omega'}=1\text{ or }\frac{1}{2}.\]
\end{lem}
\begin{proof}
In view of Proposition \ref{Delta}, and by \cite[Theorem 1.2]{DD14}, we have
$$
\frac{\Omega}{\Omega'}=\left|\frac{\omega}{\phi^*\omega'}\right|,
$$
where $\omega$ and $\omega'$ are the minimal differentials (unique up to signs) on $E$ and $E'$, respectively. Suppose $\phi^*\omega'=\alpha \omega$ and ${\phi'}^*\omega=\beta \omega'$ with $\alpha,\beta\in \BZ$. Since 
\[\alpha\beta\omega=\phi^*({\phi'}^*\omega)=[2]^*\omega=2\omega,\]
we see $\alpha=1$ or $2$ and the lemma follows.
\end{proof}

\subsection{Selmer groups}
Recall that $\phi:E \to E'$ is the isogeny of degree $2$ defined over $\BQ$ with kernel $E[\phi]=E[2](\BQ)$, and $\phi': E' \to E$ its dual isogeny. For any finite or infinite place $v$ of $\BQ$, we write $\BQ_v$ for the completion of $\BQ$ at $v$. Let $\kappa_{v,2}(E)$ be the image of the Kummer map
$$
\kappa_{v,2}: E(\BQ_v)/2E(\BQ_v) \hookrightarrow \RH^1(\BQ_v, E[2]).
$$
Similarly, we write
$$
\kappa_{v,\phi}(E):=\Im \left(E'(\BQ_v)/\phi(E(\BQ_v)) \hookrightarrow \RH^1(\BQ_v, E[\phi])\right);
$$
$$
\kappa_{v,\phi'}(E'):=\Im \left(E(\BQ_v)/\phi'(E'(\BQ_v)) \hookrightarrow \RH^1(\BQ_v, E'[\phi'])\right).
$$
The $2$-Selmer group over $\BQ$ is then defined by 
$$
\Sel_{2}(E):=\Ker \left(\RH^1(\BQ, E[2]) \to \bigoplus_v \RH^1(\BQ_v, E[2])/\kappa_{v,2}(E)\right).
$$
The Shafarevich--Tate group $\Sha(E)$ is defined by 
$$
\Sha(E):=\Ker \left(\RH^1(\BQ, E) \to \bigoplus_v \RH^1(\BQ_v, E)\right).
$$
Then we have the following well-known exact sequence
\[0\ra E(\BQ)/2E(\BQ)\ra \Sel_2(E)\ra\Sha(E)[2]\ra 0.\]
Similarly, we define the $\phi$-Selmer group and $\phi'$-Selmer groups over $\BQ$ by 
$$
\Sel_{\phi}(E):=\Ker \left(\RH^1(\BQ, E[\phi]) \to \bigoplus_v \RH^1(\BQ_v, E[\phi])/\kappa_{v,\phi}(E)\right);
$$
$$
\Sel_{\phi'}(E'):=\Ker \left(\RH^1(\BQ, E'[\phi']) \to \bigoplus_v \RH^1(\BQ_v, E'[\phi'])/\kappa_{v,\phi'}(E')\right).
$$

\begin{prop}\label{Selmer1}
Assume that the $2$-part of the Birch and Swinnerton-Dyer conjecture holds for $E$. Then we have
$$
\Sel_2(E)\cong \Sel_2(E') \cong \BZ/2\BZ; \ \ \Sel_\phi(E) \cong \Sel_{\phi'}(E') \cong \BZ/2\BZ. 
$$
\end{prop}
\begin{proof}
By Proposition \ref{L1}, we have $L(E,1)\neq 0$. By the work of Gross and Zagier \cite{GZ1986} and Kolyvagin \cite{Kolyvagin1990}, $E(\BQ)$ has rank zero and $|\Sha(E)|$ is finite, and hence a square. By Corollary \ref{24E}, we have 
\[\ord_2\left(\frac{|\Sha(E)|}{|E(\BQ)_\tor|^2}\prod_\ell c_\ell(E)\right)=-1,\]
we then must have $\Sha(E)[2]=0$, and hence $\Sel_2(E)=E[2](\BQ)=\BZ/2\BZ$. By the invariance of the Birch and Swinnerton-Dyer conjecture under isogeny  \cite[(Corollary to) Theorem 1.3]{Cassels65}, we have that  $E'(\BQ)=0$ also has rank $0$, and
\begin{eqnarray}\label{tama-E'}
\ord_2\left(\frac{|\Sha(E')|}{|E'(\BQ)_\tor|^2}\prod_\ell c_\ell(E')\right)&=&\ord_2\left(\frac{|\Sha(E)|}{|E(\BQ)_\tor|^2}\prod_\ell c_\ell(E)\right)+\ord_2\left(\frac{\Omega}{\Omega'}\right)\\
&=&-1+\ord_2\left(\frac{\Omega}{\Omega'}\right)=-1\text{ or }-2.\nonumber
\end{eqnarray}
The last equality follows from Lemma \ref{period}. In view of \eqref{tama-E'}, we have $\ord_2(|\Sha(E')|) \leq 1$. Since the order of $\Sha(E')$ is a square, we conclude that $\Sha(E')[2]=0$, and hence 
$$
\Sel_2(E')=E'[2](\BQ)=\BZ/2\BZ.
$$

For the second assertion, consider the following well-known exact sequence:
$$
0 \rightarrow E'(\BQ)[\phi']/\phi(E(\BQ)[2]) \rightarrow \Sel_{\phi}(E) \rightarrow \Sel_{2}(E) \rightarrow \Sel_{\phi'}(E') \rightarrow\Sha(E')[\phi']/\phi(\Sha(E)[2]) \rightarrow 0.
$$
Since $\Sel_2(E)=\Sel_2(E')=\BZ/2\BZ$, the non-trivial element in each Selmer group comes from the rational $2$-torsion point in $E(\BQ)[2]$ and $E'(\BQ)[2]$, respectively. Thus, we have 
$$
\Sha(E)[2]=\Sha(E')[2]=0,
$$
and hence the last term in the above exact sequence is zero. 
Since 
$$
E'(\BQ)[\phi']/\phi(E(\BQ)[2])\cong \BZ/2\BZ,
$$ 
it follows from the above exact sequence and its dual form that 
$$
\Sel_\phi(E) \cong \Sel_{\phi'}(E')\cong \Sel_2(E)\cong \Sel_2(E') \cong \BZ/2\BZ,
$$ 
as desired.
\end{proof}

\bigskip

\section{Heegner points}
Throughout this section, we let $E$, $p$ and $M$ be as in Theorem \ref{main1}, and take $K=\BQ(\sqrt{-p})$. We will construct Heegner points lying in $E^{(-pM)}(\BQ)$. In the main result of this section (Theorem \ref{divisibility}), we prove our Heegner points do indeed have infinite order, and we also establish the exact $2$-divisibility of these Heegner points. Then we show that Theorem \ref{main1} follows immediately from the non-triviality of Heegner points via the work of Gross--Zagier and Kolyvagin. The exact $2$-divisibility of Heegner points is used to prove $2$-part of the Birch and Swinnerton-Dyer conjecture for the related elliptic curves in the next section.

The induction arguments on Heegner points given in this section strengthen those in the earlier papers \cite{Tian14}, \cite{CLTZ15} and \cite{CLW16}, in two significant ways. Firstly, in all such arguments,  the non-triviality of Heegner points relies on the crucial Galois identity of the genus Heegner point given in Theorem \ref{genus}:
\[z_0^{\sigma_{t_0}}+\ov{z_0}=T,\]
where $T$ is a torsion point. What is new in our approach is that, while the  previous authors always make congruence restrictions on the prime
factors of $M$ to ensure the $2$-primary part of $T$ is non-trivial,  we make no such congruence restrictions,   allowing  the $2$-primary part of $T$ to be zero in the induction process (see Corollary \ref{split} and (\ref{novelty})). 
Secondly, all these methods use norm relations on Heegner points in one way or another. We use the full strength of these norm relations of Heegner points in the induction process to make the required $2$-adic valuation of the Hecke eigenvalues $a_q$ as small as possible.

\subsection{Non-triviality of Heegner points}
Let $\CO_K$ be the ring of integers of $K$. Since $(E,K)$ satisfies the Heegner hypothesis, there exists an ideal $\sN\subset \CO_K$ such that $\CO_K/\sN\simeq \BZ/N\BZ$. For any integer $c$ coprime to $pN$, let $\CO_c\subset \CO_K$ be the order of discriminant $-pc^2$. Define the Heegner points
$$y_c=(\BC/\CO_c\ra \BC/(\CO_c\cap \sN)^{-1})\in X_0(N)(H_c),$$
where $H_c$ denotes the ring class field over $K$ of conductor $c$.

For any abelian group $G$ denote 
$$
\wh{G}:=G\otimes_\BZ\wh{\BZ},
$$ 
where $\wh{\BZ}=\prod_{p}\BZ_p$. Let $\sigma: K^\times \backslash \wh{K}^\times \ra \Gal(K^\ab/K)$ be the Artin reciprocity map. Let $t_0\in \wh{K}^\times$ be the id\'ele such that $t_0\wh{\CO_K}=\wh{\sN}$. Let $w$ be the Atkin--Lehner operator on $X_0(N)$. We have the following proposition from \cite[\S 5]{Gross84}:
\begin{prop}\label{Galois}
The Heegner point $y_c$ satisfies 
$$y_c^w={\ov{y_c}}^{\sigma_{t_0^{-1}}}.$$
\end{prop}

Let $q_1,\cdots,q_r$ be distinct admissible primes for $E$, and set $M=q_1^*\cdots q_r^*$.  Let 
$$
H_0=K(\sqrt{q_1^*},\cdots,\sqrt{q_r^*})\subset H_M
$$ 
be the genus field. We define the genus Heegner point
$$z_0:=\tr_{H_M/H_0}f(y_M)\in E(H_0).$$

\begin{thm}\label{genus}
The genus Heegner point $z_0$ satisfies
\[z_0^{\sigma_{t_0}}+\ov{z_0}=T,\]
where $T=[H_M:H_0]f([0])\in E(\BQ)$. 
\end{thm}
\begin{proof}
By Proposition \ref{Galois} we have
\begin{eqnarray*}
T&=&\sum_{\sigma\in \Gal(H_M/H_0)} f([0])^\sigma=\sum_{\sigma\in \Gal(H_M/H_0)} (f(y_M^w)+f(y_M))^\sigma\\
&=&\sum_{\sigma\in \Gal(H_M/H_0)} (f(\ov{y_M}^{\sigma_{t_0^{-1}}})+f(y_M))^\sigma\\
&=&\ov{z_0}^{\sigma_{t_0^{-1}}}+z_0.
\end{eqnarray*}
The theorem follows by acting $\sigma_{t_0}$ on the above equality and noting that $T^{\sigma_{t_0}}=T$.
\end{proof}

\begin{lem}\label{tor}
We have $E(H_0)[2^\infty]=E(\BQ)[2].$
\end{lem}
\begin{proof}
Suppose the contrary and $P\in E(H_0)[2^\infty]\backslash E(\BQ)[2]$. Since $E$ has good reduction outside $N$, it follows that $\BQ(P)$ is unramified outside $2N$. Then we have $\BQ(P)=\BQ$ by noting $H_0=\BQ(\sqrt{-p},\sqrt{q_1^*},\cdots,\sqrt{q_r^*})$, which is a contradiction.
\end{proof}

\begin{defn}\label{prime-class}
An admissible prime $q$ is of first kind, if $q \equiv 1 \mod 4$ is inert in $K$, or $q \equiv 3 \mod 4$ splits in $K$; otherwise, $q$ is of second kind.
\end{defn}

\begin{coro}\label{split}
If all prime factors of $M$ are of first kind, then the two primary component $T_2$ of $T$ is of order $2$ and $z_0$ is of infinite order; otherwise $T_2=0$.
\end{coro}
\begin{proof}
Just note that the degree
\[[H_M:H_0]=\prod_{\begin{subarray}{c}q \mid M\\ \text{inert in }K\end{subarray}}\frac{q+1}{2}\cdot \prod_{\begin{subarray}{c}q \mid M\\ \text{split in }K\end{subarray}}\frac{q-1}{2}.\]
It follows that $[H_M:H_0]$ is odd, if all prime factors of $M$ are of first kind. By Proposition \ref{tor}, $E[2^\infty](\BQ)=E[2](\BQ)$ and $f([0])\not\in 2E(\BQ)$, we see that $T_2$ is of order two. Otherwise, $T_2=0$. 

In the case $T_2$ is of order $2$, we suppose $z_0$ is torsion, and choose sufficiently large odd $C$ so that $Cz_0\in E(H_0)[2^\infty]=E(\BQ)[2]$ and $CT=T_2$. Then 
\begin{equation}\label{simplification}
T_2=CT=(Cz_0^{\sigma_{t_0}}+\ov{Cz_0})=2Cz_0=0,
\end{equation}
which is a contradiction. Hence in this case $z_0$ is of infinite order.
\end{proof}

We only consider divisors of $M$ which are {\em congruent to $1 \mod 4$} so that $\sqrt{d}\in H_0$. For any such $d\mid M$, let $\chi_d:\Gal(H_0/K)\ra \mu_2$ be the quadratic character associated to the quadratic extension $K(\sqrt{d})/K$ through class field theory. Then the character $\chi_d$ is explicitly given by $\chi_d(\sigma)=(\sqrt{d})^{\sigma-1}$. Define the twisted Heegner points
\begin{equation}\label{z_d}
z_d=\sum_{\sigma\in \Gal(H_M/K)} \chi_d(\sigma)f(y_M)^\sigma=\sum_{\sigma\in \Gal(H_0/K)} \chi_d(\sigma )z_0^{\sigma}\in E(K(\sqrt{d}))^-,
\end{equation}
where the superscript ``$-$" means the eigenspace that the generator of $\Gal(K(\sqrt{d})/K)$ acts as $-1$.

\begin{prop}\label{relation}
We have
\[\ov{z_d}+z_d^{\sigma_{t_0}}=0,\]
and 
\[\sum_{\begin{subarray}{c}d \mid M\\ d\equiv 1\mod 4\end{subarray}} z_d=2^rz_0.\]
\end{prop}
\begin{proof}
The proof of the first assertion is similar to that of Theorem \ref{genus}. By Proposition \ref{Galois} we have
\begin{eqnarray*}
\sum_{\sigma\in \Gal(H_M/K)} \chi_d(\sigma)f([0])^\sigma&=&\sum_{\sigma\in \Gal(H_M/K)} \chi_d(\sigma)(f(y_M^w)+f(y_M))^\sigma\\
&=&\sum_{\sigma\in \Gal(H_M/K)} \chi_d(\sigma)(f(\ov{y_M}^{\sigma_{t_0^{-1}}})+f(y_M))^\sigma\\
&=&\ov{z_d}^{\sigma_{t_0^{-1}}}+z_d.
\end{eqnarray*}
On the other hand, we have 
\[\sum_{\sigma\in \Gal(H_M/K)} \chi_d(\sigma)f([0])^\sigma=\left(\sum_{\sigma}\chi_d(\sigma)\right)f([0])=0.\]
The second assertion is standard and straight-forward.
\end{proof}
\begin{thm}\label{divisibility}
The Heegner point $z_M$ has exact $2$-divisibility index $r-1$: 
\[z_M\in \left(2^{r-1}E(\BQ(\sqrt{-pM}))^-+E(\BQ(\sqrt{-pM}))_\tor \right)\left \backslash \left(2^rE(\BQ(\sqrt{-pM}))^-+E(\BQ(\sqrt{-pM}))_\tor \right),\right.\]
where the superscript ``$-$" means the eigenspace that the generator of $\Gal(\BQ(\sqrt{-pM})/\BQ)$ acts as $-1$.
\end{thm}
\begin{remark}
If $r=0$, i.e., $M=1$, the statement in the theorem exactly means $$2z_1\in E(\BQ(\sqrt{-p}))^- \text{ but }z_1 \not \in \left(E(\BQ(\sqrt{-p}))^-+E(\BQ(\sqrt{-p}))^-_\tor \right).$$
\end{remark}
\begin{proof}
In the following, without loss of generality we may replace $S=f([0])$, if necessary, by its two primary component $S_2\in E(\BQ)[2]$ for simplification. Indeed, we may do this as in (\ref{simplification})  by   multiplying all points by a sufficiently large odd integer. Let $r(M)$ denote the number of distinct prime factors of $M$. We  proceed by  induction on $r=r(M)$ to prove 
\begin{equation}\label{ind}
z_M=2^r x_M \text{ for some $x_M\in E(H_0)$ with $\ov{x_M}+x_M^{\sigma_{t_0}}=S$. }
\end{equation}
And then we deduce the assertion in the theorem from this property (\ref{ind}).

First we consider the initial case $r=0$ i.e., $M=1$. Then $H_0=K$ and $z_M=z_0$ and take $x_M=z_M\in E(K)$. By Theorem \ref{genus} we have 
\begin{equation}\label{ind0}
\ov{z_M}+{z_M}^{\sigma_{t_0}}=S.
\end{equation}
We prove the theorem in this case $M=1$ as follows. First, note that $2\ov{z_M}=-2z_M$ and hence $2z_1\in E(\BQ(\sqrt{-p}))^-$. Suppose the contrary that $z_M=y+s$ with $y\in  E(\BQ(\sqrt{-p}))^-$ and $s\in E(\BQ(\sqrt{-p}))^-_\tor$. By Lemma \ref{tor}, we can choose $C$  a sufficiently large odd integer such that $Cs\in E(\BQ)[2]$. Then from (\ref{ind0}) we have 
\[S=CS=C(z_M+\ov{z_M})=C(y+\ov y)+2Cs=2Cs=0,\]
which is a contradiction.

Next assume $r>0$. For any proper  $d\mid M$ with $d\equiv 1\mod 4$, let $d'=M/d$ and $d'^+$ (respectively, $d'^-$) be the product of prime factors of $d'$ that split (respectively, are inert) in $K$.  By Euler system property \cite{Kolyvagin1990}, we have
\[z_d=\tr_{H_M/H_{d}} f(y_M)=\left(\prod_{q\mid d'^+} (a_q-\sigma_v-\sigma_{\ov v})\cdot \prod_{q\mid d'^-} a_q\right)z_d',\]
where 
\[z_d'=\sum_{\sigma\in \Gal(H_{d}/K)}\chi_d(\sigma)f(y_d).\]
Here $v$ and $\ov v$ are the two places of $K$ above $q$ if $q \mid d'^+$, and $\sigma_v$ and $\sigma_{\ov v}$ are the corresponding Frobenius automorphisms, respectively.  
Note that we have \[\chi_d(\sigma_v)=\chi_d(\sigma_{\ov v})=\left(\frac{d}{q}\right)=\pm 1,\] 
and
\[\sigma_v(z_d')=\chi_d(\sigma_v)z_d'  \text{ \ and \ } \sigma_{\ov{v}}(z_d')=\chi_d(\sigma_{\ov v})z_d'.\]
By the condition (\ref{Hecke}) on Hecke eigenvalues, we know for all $q\mid M$ we have 
\[a_q\equiv 1-\left(\frac{-1}{q}\right)\mod 4.\]

For any proper divisor $d\mid M$ with $d\equiv 1\mod 4$, we put 
\[A_d=2^{-\mu(d')}\left(\prod_{q\mid d'^+} (a_q-2\left(\frac{d}{q}\right))\cdot \prod_{q\mid d'^-} a_q\right).\]
Let $M_1$ (respectively, $M_2$) be the product of prime factors of $M$ which are of first (respectively, second) kind.
From the condition (\ref{Hecke}) for Hecke eigenvalues and Definition \ref{prime-class}, we see the following two facts:

(\romannumeral1) If some prime factor of $d'$ is of first kind, i.e., $M_1\nmid d$, then $A_d$ is an even integer.

(\romannumeral2) If all prime factors of $d'$ are of second kind, i.e., $M_1\mid d$, then $A_d$ is an odd integer.

\noindent By the induction hypothesis, we have 
\[z_d'=2^{\mu(d)}x_d' \in 2^{\mu(d)}E(H_0),\]
with $\ov{x_d'}+x_d'^{\sigma_{t_0}}=S$.
Hence 
\[z_d=2^rA_d x_d'\in 2^r E(H_0).\]
Then 
\[z_M=2^r\left(z_0-\sum_{\begin{subarray}{c}M_1 \nmid d\\ d\neq M\end{subarray}}A_dx_d'-\sum_{\begin{subarray}{c}M_1 \mid d\\ d\neq M\end{subarray}}A_dx_d'\right)\in 2^rE(H_0).\]
Denote
\[x_M=z_0-\sum_{\begin{subarray}{c}M_1 \nmid d\\ d\neq M\end{subarray}}A_dx_d'-\sum_{\begin{subarray}{c}M_1 \mid d\\ d\neq M\end{subarray}}A_dx_d'\in E(H_0).\]
If $M_1\nmid d$, then $A_d$ is even and
\[A_d(\ov{x_d'}+{x_d'}^{\sigma_{t_0}})=0.\]
If $M_1\mid d$, then $A_d$ is odd and
\[A_d(\ov{x_d'}+{x_d'}^{\sigma_{t_0}})=S.\]
Then 
\begin{equation*}
\ov{x_M}+x_M^{\sigma_{t_0}}=\ov{z_0}+z_0^{\sigma_{t_0}}+\sum_{\begin{subarray}{c}M_1 \mid d\\ d\neq M\end{subarray}}A_d(\ov{x_d'}+\ov{x_d'}^{\sigma_{t_0}}).
\end{equation*}
If all prime factors of $M$ are of first kind, i.e., $M_1=M$ and $M_2=1$, then there are no terms with $M_1\mid d$ and by Corollary \ref{split}, we have 
\[\ov{x_M}+x_M^{\sigma_{t_0}}=\ov{z_0}+z_0^{\sigma_{t_0}}=S.\]
If some prime factors of $M$ are of second kind, again by Corollary \ref{split}, we have
\begin{equation}\label{novelty}
\ov{z_0}+z_0^{\sigma_{t_0}}=0,
\end{equation}
and hence
\begin{equation*}
\ov{x_M}+x_M^{\sigma_{t_0}}=\sum_{\begin{subarray}{c}M_1 \mid d\\ d\neq M\end{subarray}}A_d(\ov{x_d'}+\ov{x_d'}^{\sigma_{t_0}})=\sum_{\begin{subarray}{c}M_1 \mid d\\ d\neq M\end{subarray}} S=S.
\end{equation*}
In conclusion, we have 
\begin{equation}\label{x_M}
z_M=2^rx_M,\quad \ov{x_M}+x_M^{\sigma_{t_0}}=S.
\end{equation}

Next we use a descent argument (as in \cite[page 365]{CLTZ15}) to prove the theorem when $r>0$.
We embed  $E(K(\sqrt{M}))/2^r E(K(\sqrt{M}))$ into $\RH^1(K(\sqrt{M}),E[2^r])$ by the Kummer map and consider the inflation-restriction exact sequence:
\[0\ra\RH^1(H_0/K(\sqrt{M}), E(H_0)[2^r]) \ra \RH^1(K(\sqrt{M}),E[2^r])\ra \RH^1(H_0,E[2^r]).\]
Since $z_M\in 2^rE(H_0)$, the image of $z_M$ in $\RH^1(H_0,E[2^r])$ under the Kummer map is zero, so  the image of $z_M$ in $\RH^1(K(\sqrt{M}),E[2^r])$ lies in $\RH^1(H_0/K, E(H_0)[2^r]) $. Because \[E(H_0)[2^\infty]=E(\BQ)[2],\] 
the image of $z_M$ in $\RH^1(K(\sqrt{M}),E[2^r])$  must be of order $2$, that is, $2z_M\in 2^rE(K(\sqrt{M}))$.

Next we prove
\begin{equation} \label{right-place}
2z_M\in 2^rE(\BQ(\sqrt{-pM}))^-.
\end{equation}
Let $\sigma_0$ be the generator of $\Gal(K(\sqrt{M})/K)$, and by definition (\ref{z_d}) we have
\begin{equation}\label{where}
\sigma_0(z_M)=-z_M.
\end{equation}
By \cite[Theorem 5.2]{Zhaint}, we already know that the $L(E^{(M)},1)\neq 0$ and hence its root number $\left(\frac{M}{-N}\right)=1$. Note $\chi_M(\sigma_{t_0})=\left(\frac{M}{N}\right)$.  Then
\[\chi_M(\sigma_{t_0})\sgn(M)=\left(\frac{M}{-N}\right)=1.\]
Here we divide into two cases:

\noindent {\em $\bullet$ $\chi_M(\sigma_{t_0})=\sgn(M)=+1.$} Then
\[\ov{2z_M}=-\chi_M(\sigma_{t_0})2z_M=-2z_M.\]
Since $M>0$, together with (\ref{where}) we conclude\[2z_M\in 2^rE(\BQ(\sqrt{-pM
}))^-.\]
Here we note $\sigma_0$ is a generator of $\Gal(\BQ(\sqrt{-pM})/\BQ)$.

\noindent {\em $\bullet$ $\chi_M(\sigma_{t_0})=\sgn(M)=-1.$} Then
\[\ov{2z_M}=-\chi_M(\sigma_{t_0})2z_M=2z_M.\]
Again since $M<0$, we have 
\[2z_M\in 2^rE(\BQ(\sqrt{-pM}))^-.\]
Consequently we conclude (\ref{right-place}) and hence
\begin{equation*} \label{}
z_M\in 2^{r-1}E(\BQ(\sqrt{-pM}))^-+E(\BQ(\sqrt{-pM}))_\tor.
\end{equation*}

Assume the contrary that 
\[z_M=2^ry+t,\quad y\in E(K(\sqrt{-pM}))^-,t\in E(K(\sqrt{-pM}))_\tor.\]
From (\ref{x_M}), we have $x_M=y+t'$ with $t'\in E(K(\sqrt{-pM}))_\tor$ and by Lemma \ref{tor} and the fact that $a\mapsto {\ov a}^{ \sigma_{t_0}}$ is also a generator of $\Gal(\BQ(\sqrt{-pM})/\BQ)$.
\[S=\ov{x_M}+x_M^{\sigma_{t_0}}=\ov{y}+y^{\sigma_{t_0}}+2t'=2t'\]
will have trivial $2$-primary component which is a contradiction.
\end{proof}

\subsection{Explicit Gross--Zagier formulae}
Let $\pi$ be the automorphic representation on $\GL_2(\BA)$ associated to $E$, let $\pi_{\chi_M}$ be the representation on $\GL_2(\BA)$ constructed from the quadratic character $\chi_M$ on $\BA_K^\times$ through Weil--Deligne representations. Since $(E,K)$ satisfies Heegner hypothesis, the Rankin--Selberg $L$-series $L(\pi\times \pi_{\chi_M},s)$ has sign $-1$ in its functional equation. From the Artin formalism, we have 
\begin{equation}\label{decomposition}
L(\pi\times \pi_{\chi_M},s)=L(E^{(M)},s)L(E^{(-pM)},s).
\end{equation}
Here we normalize the Rankin--Selberg $L$-series with the central point at $s=1$.

Recall $f:X_0(N)\ra E$ is the optimal modular parametrization of $E$ sending $[\infty]$ to the zero element of $E$, and $g$ is the newform associated to $E$. Denote the Peterson norm
\[(g,g)_{\Gamma_0(N)}:=\int\int_{\Gamma_0(N)\backslash \CH}|g(z)|^2dxdy,\quad z=x+iy,\]
where $\CH$ is the Poincar\'e upper half plane. Invoking the general explicit Gross--Zagier formula established in \cite{CST14}, the Heegner point $z_M\in E^{(-pM)}(\BQ)$ satisfies the following explicit height formula. \begin{thm}\label{GZ}
The Heegner point $z_M\in E^{(-pM)}(\BQ)$ satisfies 
\[L'(\pi\times \pi_{\chi_M},1)=\frac{16\pi^2(g,g)_{\Gamma_0(N)}}{\sqrt{p}|M|}\cdot \frac{\wh{h}_{\BQ}(z_M)}{\deg f},\]
where $\wh{h}_\BQ(\cdot )$ denotes the N\'eron--Tate height on $E$ over $\BQ$.
\end{thm}
\begin{proof}
Since the pair $(E,K)$ satisfies the Heegner hypothesis, we just apply \cite[Theorem 1.1]{CST14}.
\end{proof}

Now we can give the proof of Theorem \ref{main1}.
\begin{thm}[Theorem \ref{main1}]\label{main1'}
Let $E$, $p$ and $M$ be as in Theorem \ref{main1}. Then we have 
$$
\ord_{s=1}L(E^{(M)}, s)=\rk\  E^{(M)}(\BQ)=0 \text{ \ and \ } 
\ord_{s=1}L(E^{(-pM)}, s)=\rk\  E^{(-pM)}(\BQ)=1.
$$
Moreover, the Shafarevich--Tate groups $\Sha(E^{(M)})$ and $\Sha(E^{(-pM)})$ are finite.
\end{thm}
\begin{proof}
By Theorem \ref{divisibility}, $z_M$ has infinite order and hence $L(\pi\times \pi_{\chi_M},s)$ has vanishing order $1$ at $s=1$.  From the decomposition in \eqref{decomposition} we must have
\[\ord_{s=1}L(E^{(M)},s)=0 \text{ \ and \ } \ord_{s=1}L(E^{(-pM)},s)=1.\]
Then Theorem \ref{main1} follows from the work of Gross--Zagier and Kolyvagin.
\end{proof}

Let $\Omega_{E^{(D)}}^+$ be the minimal real period of ${E^{(D)}}$. Let $\Omega^{(D)}$ denote the period of $E^{(D)}$, which is equal to $\Omega_{E^{(D)}}^+$ or $2\Omega_{E^{(D)}}^+$ according to that $E(\BR)$ is connected or not. 
\begin{prop}\label{period-relation}
We have the following relation between periods and Peterson norm of $g$:
\[\Omega^{(M)}\Omega^{(-pM)}=\frac{8\pi^2(g,g)_{\Gamma_0(N)}}{\deg f\sqrt{p}|M|}\]
\end{prop}
\begin{proof}
Since the discriminant $\Delta_E<0$, $E(\BR)$ has only one connected component. The period lattice $L=\BZ\Omega+\BZ(\Omega/2+\Omega^-/2)$. We have the following period relation:
\[\frac{8\pi^2(g,g)_{\Gamma_0(N)}}{\deg f\sqrt{p}|M|}=\frac{2}{\sqrt{p}|M|}\int_{\BC/L}dxdy=\frac{\Omega\Omega^-}{\sqrt{-p}|M|}=\Omega^{(M)}\Omega^{(-pM)},\]
where the last equality follows from \cite{Pal12}.
\end{proof}

Combining Theorem \ref{GZ}, Proposition \ref{period-relation} and the decomposition $(\ref{decomposition})$, we have
\begin{coro}\label{GZ1}
The Heegner point $z_M\in E^{(-pM)}(\BQ)$ satisfies 
\[\frac{L(E^{(M)},1)L'(E^{(-pM)},1)}{\Omega^{(M)}\Omega^{(-pM)}}=2\wh{h}_{\BQ}(z_M).\]
\end{coro}

\bigskip

\section{The $2$-part of the Birch and Swinnerton-Dyer conjecture for the quadratic twists}
Throughout this section, we assume $E$ has odd Manin constant. We shall compare all the arithmetic invariants of $E$ with those of $E^{(M)}$ and $E^{(-pM)}$.

\subsection{Tamagawa numbers}
\begin{prop}\label{Tamagawa}
Let $q$ be any prime dividing $D$ with $(D,2N)=1$. We have 
$$
c_q(E^{(D)}) = 
\left\{
\begin{array}{ll}
2 & \hbox{if $q$ is inert in $\BQ(\sqrt{\Delta_E})$;} \\
4 & \hbox{if $q$ splits in $\BQ(\sqrt{\Delta_E})$.}
\end{array}
\right.
$$
\end{prop}
\begin{proof}
Since $(q, 2N)=1$, reduction modulo $q$ on $E$ gives an isomorphism 
$$
E(\BQ_q)[2] \cong E(\BF_q)[2].
$$ 
It follows from \cite[Lemma 36 and Lemma 37]{Coates13} that 
$$
\ord_2(c_q(E^{(D)})) = \ord_2(|E^{(D)}(\BQ_q)[2]|) = \ord_2(|E(\BQ_q)[2]|) = \ord_2(|E(\BF_q)[2]|),
$$
which is equal to $1$ if $q$ is inert in $\BQ(E[2])$, and equal to $2$ if $q$ splits in $\BQ(E[2])$. Note that the curve $E^{(D)}$ has additive reduction at $q$, so we have $c_q(E^{(D)}) \leq 4$. The assertion of the lemma then follows.
\end{proof}

\begin{prop}\label{Tamagawa2}
Let $D\equiv 1\mod 4$ be a square-free integer with $(D, N)=1$. Then
$$
\ord_2\left(\prod_{\ell \mid N} c_\ell (E^{(D)})\right)=1.
$$
\end{prop}\label{Tamagawa3}
\begin{proof}
By Proposition \ref{24E}, we have 
\[\ord_2\left(\frac{|\Sha(E)|}{|E(\BQ)_\tor|^2}\prod_\ell c_\ell(E)\right)=-1,\]
Since $\ord_2(|E(\BQ)_\tor|^2)=2$ and $|\Sha(E)|$ must be a square, it follows that $\Sha(E)[2]=0$ and 
\begin{equation}\label{tama-E}
\ord_2\left(\prod_{\ell \mid N} c_\ell (E)\right)=1.
\end{equation}
We now prove the assertion in several cases according to the reduction types of $E$. 

(\romannumeral1) If $E$, and hence $E^{(M)}$, have additive reduction at $\ell$, from the usual table of special fibers the $2$-primary part of the connected component group is killed by $2$. Again, by \cite[Lemma 36 and Lemma 37]{Coates13}, we have 
\begin{equation}\label{add}
\ord_2(c_\ell(E))=\ord_2(|E(\BQ_\ell)[2]|)=\ord_2(c_\ell(E^{(M)})).
\end{equation}

(\romannumeral2) If $E$ has multiplicative reduction at $\ell$, and $\ell \mid N$ splits in $\BQ(\sqrt{D})$. Then $E^{(D)}$ is isomorphic to $E$ over $\BQ_\ell$, we immediately have $c_\ell(E^{(D)})=c_\ell(E)$. 

(\romannumeral3) If $E$ has split multiplicative reduction at $\ell$, and $\ell\mid N$ is inert in $\BQ(\sqrt{D})$.
Then $E^{(M)}$ has nonsplit multiplicative reduction at $\ell$.  In this case $E^{(M)}_{\BQ_{\ell^2}}$ and $E_{\BQ_{\ell^2}}$ are isomorphic and have the same Tamagawa number equal to $c_\ell(E)$. In view of \cite[page 366]{Silvermanbook2}, if $c_\ell(E)$ is odd (respectively, even), then $c_\ell(E^{(M)})=1$ (respectively, $2$), and hence
\[\ord_2(c_\ell(E))=\ord_2(c_\ell(E^{(M)})).\]

(\romannumeral4) If $E$ has nonsplit multiplicative reduction at $\ell$, and $\ell\mid N$ is inert in $\BQ(\sqrt{D})$. Then $E^{(M)}$ has split multiplicative reduction at $\ell$. If $c_\ell(E)=1$, then both $c_\ell(E)$ and $c_\ell(E^{(M)})$ are odd, i.e., 
\[\ord_2(c_\ell(E))=\ord_2(c_\ell(E^{(M)}))=0.\]
If $c_\ell(E)=2$, then $E$ is semi-stable. Indeed from (\ref{add}), the Tamagawa number is even for places of additive reduction. But there is only one bad place of $E$ with even Tamagawa number. By (\ref{tama-E'}) and (\ref{tama-E}), the $2$-adic valuation of the product of all Tamagawa numbers of $E$ (respectively, $E'$)  is $1$ (respectively, $\leq 1$). Since $E$ is semi-stable of root number $+1$ and $c_\ell(E)=2$, we conclude from \cite[Theorem 6.1]{DD14} that $c_\ell(E')=1$. Again, we have 
$$
\ord_2(c_\ell(E'))=\ord_2(c_\ell({E'}^{(M)}))=0.
$$ 
Now $\phi^{(M)}:E^{(M)}\ra {E'}^{(M)}$ is an isogeny of degree $2$ of elliptic curves of split multiplication at $\ell$. Since $c_\ell({E'}^{(M)})$ is odd, by \cite[Theorem 6.1]{DD14}, $\ord_2(c_\ell(E^{(M)}))=1$.
Therefore, for any $\ell\mid N$, we always have
\[\ord_2(c_\ell(E))=\ord_2(c_\ell(E^{(M)})).\]
This completes the proof of this proposition.
\end{proof}

\subsection{Selmer groups} 
Let $E$, $p$ and $M$ be as in Theorem \ref{main1}. Recall that $E^{(M)}$ is the quadratic twist of $E$ by $\BQ(\sqrt{M})$. Let $\phi^{(M)}:E^{(M)}\ra E'^{(M)}$ be an isogeny of degree $2$ with kernel $E^{(M)}[2](\BQ)$ and ${\phi'}^{(M)}$ its dual isogeny. There is a natural identification of Galois modules $E[\phi]=E^{(M)}[\phi^{(M)}]$. This allows us to view both $\Sel_\phi(E)$ and $ \Sel_{\phi^{(M)}}(E^{(M)})$ as  subgroups of $\RH^1(\BQ, E[\phi])$ and enable us to compare the Selmer groups $\Sel_\phi(E)$, $\Sel_{\phi^{(M)}}(E^{(M)})$ by local Kummer conditions in $\RH^1(\BQ_\ell,E[\phi])$ for various places $\ell$ of $\BQ$.

\begin{lem}\label{Local}
If all the prime factors of $2N$ split in $\BQ(\sqrt{M})$, then 
$$
\kappa_{\ell,\phi}(E)=\kappa_{\ell,\phi^{(M)}}(E^{(M)}) \ \ \text{and} \ \ \kappa_{\ell,\phi'}(E')=\kappa_{\ell,{\phi'}^{(M)}}({E'}^{(M)})
$$ 
for all places $\ell$. 
\end{lem}
\begin{proof}
First consider the case $\ell=\infty$. We have the following Weierstrass equations
$$E^{(M)}: y^{2}=x^3+aMx^2+bM^2x,\quad 
E'^{(M)}: y^{2}=x^3-2aMx^{2}+(a^2-4b)M^2x.
$$
By Proposition \ref{dis}, $\Delta_{E^{(M)}}=M^6\Delta_E<0$ and $\Delta_{{E'}^{M}}=M^6\Delta_{E'}>0$.
Hence both $E(\BR)$ and $E^{(M)}(\BR)$ are connected while both $E'(\BR)$ and ${E'}^{(M)}(\BR)$ have two connected components. Then
\[\kappa_{\infty,\phi}(E)=\kappa_{\infty,\phi^{(M)}}(E^{(M)})=\BZ/2\BZ\] 
and \[\kappa_{\infty,\phi'}(E')=\kappa_{\infty,{\phi'}^{(M)}}({E'}^{(M)})=0.\]

Suppose $\ell\mid 2N$, then $E$ and $E^{(M)}$ are isomorphic over $\BQ_\ell$. In particular, 
$$
\kappa_{\ell,\phi}(E)=\kappa_{\ell,\phi^{(M)}}(E^{(M)}) \ \ \text{and} \ \ \kappa_{\ell,\phi'}(E')=\kappa_{\ell,{\phi'}^{(M)}}({E'}^{(M)})
$$ 
for $\ell\mid N$.  We now compare the local conditions at each place $\ell\nmid 2N\infty$ and divide it into two cases.

(\romannumeral1) For $\ell \nmid 2NM\infty$, we have both $E$ and $E'$ have good reduction at $\ell$ and $\ell$ is unramified in $\BQ(\sqrt{M})/\BQ$. 

(\romannumeral2) For $\ell|M$, we have both $E$ and $E'$ have good reduction at $\ell$ and $\ell$ is ramified in $\BQ(\sqrt{M})/\BQ$, whence we have $E(\BQ_\ell)[2] \cong E'(\BQ_\ell)[2] \cong \BZ/2\BZ$.

Applying \cite[Lemma 6.8]{Klagsbrun17},  we have 
$$
\kappa_{\ell,\phi}(E)=\kappa_{\ell,\phi^{(M)}}(E^{(M)}) \ \ \text{and} \ \ \kappa_{\ell,\phi'}(E')=\kappa_{\ell,{\phi'}^{(M)}}({E'}^{(M)})
$$
for $\ell \nmid 2N\infty$. Then the lemma follows.
\end{proof}

\begin{lem}\label{Selmer2}
If all the prime factors of $2N$ split in $\BQ(\sqrt{M})$, then 
$$
\Sel_\phi(E^{(M)}) \cong \Sel_{\phi'}(E'^{(M)}) \cong \BZ/2\BZ.
$$
\end{lem}
\begin{proof}
By Lemma \ref{Local}, we see that $\kappa_{\ell,\phi}(E)=\kappa_{\ell,\phi^{(M)}}(E^{(M)})$ holds for any places. Therefore, we have
\[ \Sel_{\phi^{(M)}}(E^{(M)})=\Sel_\phi(E)=\BZ/2\BZ.\]
The result for $E'$ and ${E'}^{(M)}$ follows similarly.
\end{proof}

\begin{prop}\label{SelmerM}
If all the prime factors of $2N$ split in $\BQ(\sqrt{M})$, then 
$$
\Sel_2(E^{(M)})=\BZ/2\BZ \ \ \text{and} \ \ \Sha(E^{(M)})[2]=0.
$$
\end{prop}
\begin{proof}
The first assertion follows from Lemma \ref{Selmer2}, and the following exact sequence:
\[0 \rightarrow E'^{(M)}(\BQ)[\phi'^{(M)}]/\phi^{(M)}(E^{(M)}(\BQ)[2]) \rightarrow \] \[\Sel_{\phi^{(M)}}(E^{(M)}) \rightarrow \Sel_{2}(E^{(M)}) \rightarrow \Sel_{\phi'^{(M)}}(E'^{(M)}),\]
Noting  $\Sel_2(E^{(M)})=E^{(M)}(\BQ)[2]=\BZ/2\BZ$, we have $\Sha(E^{(M)})[2]=0$.
\end{proof}

\begin{lem}\label{p}
The prime $p$ is inert in $\BQ(\sqrt{\Delta_E})$, and split in $\BQ(\sqrt{\Delta_{E'}})$.
\end{lem}
\begin{proof}
Since $(E,K=\BQ(\sqrt{-p}))$ satisfies the Heegner hypothesis, for all prime $\ell \mid N$, we have 
\[\left(\frac{\ell}{p}\right)=\left(\frac{-p}{\ell}\right)=+1.\]
By Proposition \ref{dis}, $\Delta_E<0$. Then, noting, modulo squares, $N$ and $\Delta_E$ have the same prime factors, it follows that 
\[\left(\frac{\Delta_E}{p}\right)=\left(\frac{-1\cdot |\Delta_E|}{p}\right)=\left(\frac{-1}{p}\right)=-1.\]
Hence $p$ is inert in $\BQ(\sqrt{\Delta_E})$. Similarly, the second assertion is also true since $\Delta_{E'}>0$ by Proposition \ref{Delta}.
\end{proof}

The following result is due to Cassels \cite{Cassels65}. 
\begin{lem}\label{Cassels}
We have 
$$
\frac{|\Sel_{\phi}(E)|}{|\Sel_{\phi'}(E')|} = \prod_\ell \frac{|\kappa_{\ell,\phi}(E)|}{2}. 
$$
\end{lem}

\begin{prop}\label{Selmer-pM}
Let $E$, $p$ and $M$ be as in Theorem \ref{main1}. If all the prime factors of $2N$ split in $\BQ(\sqrt{M})$ and $p\equiv -1\mod 8$, then $\Sel_2(E^{(-pM)})=(\BZ/2\BZ)^2$ and $\Sha(E^{(-pM)})[2]=0$.
\end{prop}
\begin{proof}
The proof is similar to that of Proposition \ref{SelmerM}. Since $(E,K=\BQ(\sqrt{-p}))$ satisfies the Heegner hypothesis, we have $\left(\frac{-p}{\ell}\right)=+1$ for all $\ell\mid N$. Moreover, since $p\equiv -1\mod 8$, it follows that $\left(\frac{-p}{2}\right)=+1$. Then we have $\left(\frac{-pM}{\ell}\right)=+1$ for all $\ell\mid 2N$, i.e., any prime $\ell\mid 2N$ is split in $\BQ(\sqrt{-pM})$. Then for any place $\ell\neq p$, as in Lemma \ref{Local}, we have 
$$
\kappa_{\ell,\phi}(E)=\kappa_{\ell,\phi^{(M)}}(E^{(M)}) \ \ \text{and} \ \ \kappa_{\ell,\phi'}(E')=\kappa_{\ell,{\phi'}^{(M)}}({E'}^{(M)}).
$$
It suffices to compare the local Kummer conditions at $p$. By Lemma \ref{p}, $p$ is inert in $\BQ(\sqrt{\Delta_E})$ and splits in $\BQ(\sqrt{\Delta_{E'}})$. Then 
$$
E[2](\BQ_p)=E[2](\BQ)\ \ \text{and} \ \ E'[2](\BQ_p)=E'[2].
$$ 
By \cite[Lemma 6.7]{Klagsbrun17}, $\kappa_{p,\phi^{(-pM)}}(E^{(-pM)})$ (respectively, $\kappa_{p,{\phi'}^{(-pM)}}({E'}^{(-pM)})$) has dimension $2$ (respectively, $0$) over $\BZ/2\BZ$, i.e., 
\[\kappa_{p,\phi^{(-pM)}}(E^{(-pM)})=\RH^1(\BQ_p,E[\phi])\ \ \text{and} \ \ \kappa_{p,{\phi'}^{(-pM)}}({E'}^{(-pM)})=0.\]
On the other hand, since $p$ is good for $E$ and $E'$ and $p$ is prime to the degrees of $\phi$ and $\phi'$, we have 
$$
\kappa_{p,\phi}(E)=\RH^1_\ur(\BQ_p,E[\phi])\ \ \text{and} \ \ \kappa_{p,\phi'}(E')=\RH^1_\ur(\BQ_p,E'[\phi']).
$$ 
By comparing the local Kummer conditions for $\phi'$ and $\phi'^{(-pM)}$, we see 
$$
\Sel_{{\phi'}^{(-pM)}}({E'}^{(-pM)})\subset \Sel_{\phi'}(E')=\BZ/2\BZ.
$$ 
From \[\BZ/2\BZ={E}^{(-pM)}(\BQ)_\tor/\phi'^{(-pM)}({E'}^{(-pM)}(\BQ)_\tor)\subset \Sel_{{\phi'}^{(-pM)}}({E'}^{(-pM)}),\]
we conclude 
$$
\Sel_{{\phi'}^{(-pM)}}({E'}^{(-pM)})= \Sel_{\phi'}(E')=\BZ/2\BZ.
$$ 

By Lemma \ref{Cassels} and Proposition \ref{Selmer1}, we have 
\[\frac{|\Sel_{{\phi}^{(-pM)}}({E}^{(-pM)})|}{|\Sel_{{\phi'}^{(-pM)}}({E'}^{(-pM)})|}=\prod_\ell \frac{|\kappa_{\ell,\phi^{(-pM)}}(E^{(-pM)})|}{2}.\]
and \[\frac{|\Sel_{{\phi}}({E})|}{|\Sel_{{\phi'}}({E'}^{})|}=\prod_\ell \frac{|\kappa_{\ell,\phi}(E)|}{2}=1.\]
Since $\kappa_{\ell,\phi^{(-pM)}}(E^{(-pM)})=\kappa_{\ell,\phi}(E)$ for $\ell\neq p$, it follows that 
\[\frac{|\Sel_{{\phi}^{(-pM)}}({E}^{(-pM)})|}{|\Sel_{{\phi'}^{(-pM)}}({E'}^{(-pM)})|}=\prod_\ell \frac{|\kappa_{\ell,\phi^{(-pM)}}(E^{(-pM)})|}{|\kappa_{\ell,\phi}(E)|}=\frac{|\kappa_{p,\phi^{(-pM)}}(E^{(-pM)})|}{|\kappa_{p,\phi}(E)|}=2.\]
Thus 
\[\Sel_{{\phi'}^{(-pM)}}({E'}^{(-pM)})=\BZ/2\BZ \ \ \text{and} \ \ \Sel_{{\phi}^{(-pM)}}({E}^{(-pM)})\simeq (\BZ/2\BZ)^2.\]
Note $E^{(-pM)}(\BQ)/2E^{(-pM)}(\BQ)=(\BZ/2\BZ)^2$. Considering the exact sequence 
\[0 \rightarrow E'^{(-pM)}(\BQ)[\phi'^{(-pM)}]/\phi^{(-pM)}(E^{(-pM)}(\BQ)[2]) \rightarrow \] \[\Sel_{\phi^{(-pM)}}(E^{(-pM)}) \rightarrow \Sel_{2}(E^{(-pM)}) \rightarrow \Sel_{\phi'^{(-pM)}}(E'^{(-pM)}),\]
we conclude
\[\Sel_{2}(E^{(-pM)}) = E^{(-pM)}(\BQ)/2E^{(-pM)}(\BQ)=(\BZ/2\BZ)^2\text{ and }\Sha(E^{(-pM)})[2]=0,\]
as desired.
\end{proof}

\begin{prop}\label{inv}
Let $E$, $p$ and $M$ be as in Theorem \ref{main1} and suppose that all the prime factors of $2N$ split in $\BQ(\sqrt{M})$ and $p\equiv -1\mod 8$. We have
\[\ord_2\left(\frac{|\Sha(E^{(M)})|}{|E^{(M)}(\BQ)_\tor|^2}\cdot\prod_{\ell}c_\ell(E^{(M)})\right)=r-1,\]
and
\[\ord_2\left(\frac{|\Sha(E^{(-pM)})|}{|E^{(-pM)}(\BQ)_\tor|^2}\cdot\prod_{\ell}c_\ell(E^{(-pM)})\right)=r.\]
\end{prop}
\begin{proof}
Note that
\[E^{(M)}(\BQ)[2^\infty]=E^{(-pM)}(\BQ)[2^\infty]=E(\BQ)[2^\infty]=\BZ/2\BZ,\]
and by Proposition \ref{SelmerM} and \ref{Selmer-pM} we have 
\[\Sha(E^{(M)})[2^\infty]=\Sha(E^{(-pM)})[2^\infty]=\Sha(E)[2^\infty]=0.\]

Since, by assumption, any prime $q\mid M$ and, by Lemma \ref{p}, the prime $p$ are inert in $\BQ(\sqrt{\Delta_E})$,  it follows from Proposition \ref{Tamagawa} that
\[c_q(E^{(M)})=c_q(E^{(-pM)})=2 \text{ \ and \ } c_p(E^{(-pM)})=2.\]
Then the assertion of the proposition follows from Proposition \ref{Tamagawa2}.
\end{proof}

\subsection{The $2$-part of the Birch and Swinnerton-Dyer conjecture}
Recall that, for any elliptic curve defined over $\BQ$, the Birch and Swinnerton-Dyer conjecture predicts that
\begin{equation}\label{fbsd}
\frac{L^{(r_\an)}(E,1)}{{r_\an}!\Omega_E R(E) }= \frac{\prod_\ell c_\ell(E)\cdot |\Sha(E)|}{|E(\BQ)_\mathrm{tor}|^2},  
\end{equation}
where $r_\an:=\ord_{s=1}L(E,s)$, and $R(E)$ is the regulator formed with the N\'{e}ron--Tate pairing. It is known that $R(E)=1$ when $r_\an=0$, and $R(E)=\wh{h}_\BQ(P)$ when $r_\an=1$, where $P$ is a free generator of $E(\BQ)$, and $\wh{h}_\BQ(P)$ is the N\'eron--Tate height on $E$ over $\BQ$. At present, the finiteness of $\Sha(E)$ is only known when $r_\an$ is at most $1$, in which case it is also known that $r_\an$ is equal to the rank of $E(\BQ)$. Suppose the $2$-part of \eqref{fbsd} holds for $E$, we shall show that the $2$-part of \eqref{fbsd} holds for both $E^{(M)}$ and $E^{(-pM)}$ under mild assumptions.

\begin{thm}\label{2BSD}
Let $E$, $p$ and $M$ be as in Theorem \ref{main1} and suppose that all the prime factors of $2N$ split in $\BQ(\sqrt{M})$,  $p\equiv -1\mod 8$ and $E$ has odd Manin constant.
Then we have
$$
\ord_{s=1}L(E^{(M)},s)=\rk E^{(M)}(\BQ)=0, \text{ \ and \ } \ord_{s=1}L(E^{(-pM)},s)=\rk E^{(-pM)}(\BQ)=1;
$$
$$
\ord_2(L(E^{(M)},1)/\Omega_{E^{(M)}})=r-1, \text{ \ and \ } \ord_2(L'(E^{(-pM)},1)/\Omega_{E^{(-pM)}}R(E^{(-pM)}))=r.
$$
Moreover, the Shafarevich--Tate groups $\Sha(E^{(M)})$ and $\Sha(E^{(-pM)})$ are both finite of odd cardinalities. If the $2$-part of the Birch and Swinnerton-Dyer conjecture holds for $E$, then the $2$-part of the Birch and Swinnerton-Dyer conjecture holds for both $E^{(M)}$ and $E^{(-pM)}$.
\end{thm}
\begin{proof}
The rank part of the Birch and Swinnerton-Dyer conjecture for the two elliptic curves $E^{(M)}$ and $E^{(-pM)}$ has been established in Theorem \ref{main1'}. In particular, by the work of Gross--Zagier and Kolyvain, the Shafarevich--Tate groups $\Sha(E^{(M)}), \Sha(E^{(-pM)})$ are finite. Moreover, the cardinalities of both $\Sha(E^{(M)})$ and $\Sha(E^{(-pM)})$ are odd by Proposition \ref{SelmerM} and Proposition \ref{Selmer-pM}.

First consider the quadratic twists $E^{(M)}$. The full Birch and Swinnerton-Dyer conjecture predicts
\[\frac{L(E^{(M)},1)}{\Omega^{(M)}}=^?\frac{|\Sha(E^{(M)})|}{|E^{(M)}(\BQ)_\tor|^2}\cdot\prod_{\ell}c_\ell(E^{(M)}).\leqno{\mathrm{BSD(E^{(M)})}}\]
By Proposition \ref{inv} and \cite[Theorem 5.2]{Zhaint}, both sides of ${\mathrm{BSD(E^{(M)})}}$ have $2$-adic valuation $r-1$, and hence the $2$-part of the Birch and Swinnerton-Dyer conjecture for $E^{(M)}$ follows.

Next we consider the quadratic twists $E^{(-pM)}$. The full Birch and Swinnerton-Dyer conjecture predicts
\[\label{BSD2}\frac{L'(E^{(-pM)},1)}{\Omega^{(-pM)}}=^?\wh{h}_\BQ(P)\frac{|\Sha(E^{(-pM)})|}{|E^{(-pM)}(\BQ)_\tor|^2}\cdot\prod_{\ell}c_\ell(E^{(-pM)}),\leqno{\mathrm{BSD(E^{(-pM)})}} \]
where $P$ is a free generator of $E^{(-pM)}(\BQ)$. 
By the explicit height formula in Corollary \ref{GZ1}, $\mathrm{BSD(E^{(M)})}$ and $\mathrm{BSD(E^{(-pM)})}$ amount to
\begin{equation}\label{div}
\ord_2\left(\frac{\wh{h}_\BQ(z_M)}{\wh{h}_\BQ(P)}\right)=^?\ord_2\left(2^{-1}\cdot \frac{|\Sha(E^{(M)})||\Sha(E^{(-pM)})|}{|E^{(M)}(\BQ)_\tor|^2|E^{(-pM)}(\BQ)_\tor|^2}\cdot\prod_{\ell}\left(c_\ell(E^{(M)})c_\ell(E^{(-pM)})\right)\right).
\end{equation}
By Proposition \ref{inv}, the RHS is $2(r-1)$, and therefore the above equality (\ref{div}) follows from the exact $2$-divisibility of the Heegner point  $z_M$ in Theorem \ref{divisibility}. Since the $2$-part of ${\mathrm{BSD(E^{(M)})}} $ is proved, the $2$-part of ${\mathrm{BSD(E^{(-pM)})}}$ follows.
\end{proof}

\bigskip

\section{Applications}\label{eg}
In this final section, we shall illustrate our general results for the family of  quadratic twists both of the Neumann--Setzer elliptic curves, and also some elliptic curves of small conductor.

\subsection{The Neumann--Setzer elliptic curves}
Recall that the Neumann--Setzer elliptic curves have prime conductor $p_0$ (see \cite{Neumann71}, \cite{Neumann73} and \cite{Setzer75}), where $p_0$ is any prime of the form $p_0=u^2+64$ for some integer $u \equiv 1 \mod 4$, for example, $p_0=73, 89, \ldots$. Then, up to isomorphism, it is known that there are just two elliptic curves of conductor $p_0$ with a rational $2$-division point, namely,
\begin{eqnarray}\label{equa-A}
A :& y^2 + xy &= x^3 + \frac{u-1}{4}x^2 + 4x + u, \\
A':& y^2 + xy &= x^3 -\frac{u-1}{4}x^2 -x.
\end{eqnarray}
The curves $A$ and $A'$ are $2$-isogenous,  and both have Mordell--Weil groups $\BZ/2\BZ$. A simple computation shows that $\Delta_A=-p_0^2$ and $\Delta_{A'}=p_0$, it follows that
$$
\BQ(A[2])=\BQ(i), \ \BQ(A'[2])=\BQ(\sqrt{p_0}).
$$ 
Let $X_0(p_0)$ be the modular curve of level $p_0$, and there is a non-constant rational map 
$$X_0(p_0) \to A,$$ 
making the modular parametrization $\Gamma_0(p_0)$-optimal by Mestre and Oesterl\'e \cite{MO89}. Since the conductor of $A$ is odd, the Manin constant of $A$ is odd by the work of Abbes and Ullmo \cite{AU96}. In particular, we have the following result on applying our main theorem.

\begin{thm}\label{A-BSD}
Assume that $p_0$ is a prime of the form $u^2 + 64$ with $u \equiv 5 \mod 8$, and let $A$ be the Neumann--Setzer curve \eqref{equa-A}. Let $q_1,\cdots,q_r$ be distinct primes congruent to $3$ modulo $4$ which are inert in $\BQ(\sqrt{p_0})$. Let $p \equiv 3 \mod 4$ be a prime greater than $3$ which splits in $\BQ(\sqrt{p_0})$. Denote $M=q_1^*\cdots q_r^*$. Then we have 
$$
\ord_{s=1}L(A^{(M)},s)=\rk\  A^{(M)}(\BQ)=0  \text{ \ and \ } \ord_{s=1}L(A^{(-pM)},s)=\rk\  A^{(-pM)}(\BQ)=1.
$$
In particular, $\Sha(A^{(M)})$ and $\Sha(A^{(-pM)})$ are both finite of odd cardinality. Moreover, the $2$-part of the Birch and Swinnerton-Dyer conjecture is valid for both $A^{(M)}$ and $A^{(-pM)}$.
\end{thm}
\begin{proof}
The primes $q_i$,  $1 \leq i \leq r$, are admissible for $A$, since they are inert in both $\BQ(A[2])$ and $\BQ(A'[2])$. The condition on $p$ implies that $p_0$ splits in $K=\BQ(\sqrt{-p})$. Hence, the Heegner hypothesis for $(A, K)$ holds. When $u \equiv 5 \mod 8$, it follows from \cite[Lemma 5.11]{Zhai16}) that the modular parametrization $f([0])$ is precisely the non-trivial torsion point of order $2$. Thus, all the assumptions in Theorem \ref{main1} are satisfied.

Moreover, the Manin constant of $A$ is odd. The $2$-part of the Birch and Swinnerton-Dyer conjecture of $A$ is verified in \cite[Proposition 5.13]{Zhai16}. However, here we will not apply Theorem \ref{main2}, since a classical $2$-descent has been carried out in \cite[Section 5]{Zhai16}, which shows that under the assumptions in Theorem \ref{A-BSD}, we have 
$$
\Sha(A^{(M)})[2]=\Sha(A^{(-pM)})[2]=0.
$$
By Proposition \ref{Tamagawa} and Proposition \ref{Tamagawa2}, we have 
$$
\ord_2 \left(\prod_{\ell} c_{\ell} (A^{(M)})\right) = r+1 \text{ and } \ord_2 \left(\prod_{\ell} c_{\ell} (A^{(-pM)})\right) = r+2.
$$
Note that, by \cite[Theorem 5.2]{Zhaint}, we have 
$$
\ord_2(L(A^{(M)}, 1)/\Omega_{A^{(M)}})=r-1.
$$ 
Then combining with equation \eqref{div}, it follows that the $2$-part of the Birch and Swinnerton-Dyer conjecture is valid for both $A^{(M)}$ and $A^{(-pM)}$. 
\end{proof}

Note that for $u \equiv 1 \mod 8$, the same result holds if we assume $\ord_2(L^{\alg}(A,1))=-1$. For example, we can take $p_0=73$. Here is the beginning of an infinite set of primes which are congruent to $3$ modulo $4$ and inert in $\BQ(\sqrt{73})$:
$$
\mathcal{S}=\{7, 11, 31, 43, 47, 59, 83, 103, 107, 131, 139, 151, 163, 167, 179, 191, 199, \ldots \}.
$$
For more examples, there is a nice table presenting primes of the form $u^2+64$ in \cite{Setzer75}.

\subsection{More numerical examples}
The theorem can be applied on the family of quadratic twists of many elliptic curves $E/\BQ$, we include a table here when the conductor of $E$ is less than $100$.

\bigskip
\bigskip

\begin{center}
\tablefirsthead{
\multicolumn{5}{c}{{\bf Table.} $E/\BQ$ satisfying $f([0])\not\in 2E(\BQ)$ and Condition $(\mathrm{Tor})$ with conductor $N<100$.}\\
\multicolumn{5}{c}{}\\
\hline
\multicolumn{5}{|c|}{$E/\BQ$ satisfying $f([0])\not\in 2E(\BQ)$ and Condition $(\mathrm{Tor})$.}\\
\hline  $E$ & $\BQ(E[2])$ & $\BQ(E'[2])$ & admissible primes $q$ & $p$ \\
\hline }

\tablehead{
\hline  $E$ & $\BQ(E[2])$ & $\BQ(E'[2])$ & admissible primes $q$ & $p$ \\
\hline }

\tabletail{\hline}

\tablelasttail{\hline}
\begin{supertabular}{|c|c|c|c|c|}

$14a1$&$\BQ(\sqrt{-7})$&$\BQ(\sqrt{2})$&$3,5,13,19,59,61,83,101,\ldots$&$7,31,47,103,167,199,\ldots$ \\
\hline
$20a1$&$\BQ(\sqrt{-1})$&$\BQ(\sqrt{5})$&$3,7,23,43,47,67,83,103,\ldots$&$31,71,79,151,191,199,\ldots$ \\
\hline
$36a1$&$\BQ(\sqrt{-3})$&$\BQ(\sqrt{3})$&$5,17,29,41,53,89,101,113,\ldots$&$23,47,71,167,191,239,\ldots$ \\
\hline
$46a1$&$\BQ(\sqrt{-23})$&$\BQ(\sqrt{2})$&$5,11,19,37,43,53,61,67,\ldots$&$7,79,103,191,199,263,\ldots$ \\
\hline
$49a1$&$\BQ(\sqrt{-7})$&$\BQ(\sqrt{7})$&$5,13,17,41,61,73,89,97,\ldots$&$19,31,47,59,83,103,\ldots$ \\
\hline
$52a1$&$\BQ(\sqrt{-1})$&$\BQ(\sqrt{13})$&$7,11,19,31,47,59,67,71,\ldots$&$23,79,103,127,191,199,\ldots$ \\
\hline
$56b1$&$\BQ(\sqrt{-7})$&$\BQ(\sqrt{2})$&$3,5,13,19,59,61,83,101,\ldots$&$31,47,103,167,199,223,\ldots$ \\
\hline
$69a1$&$\BQ(\sqrt{-23})$&$\BQ(\sqrt{3})$&$5,7,17,19,43,53,67,79,\ldots$&$11,83,107,191,227,251,\ldots$ \\
\hline
$73a1$&$\BQ(\sqrt{-1})$&$\BQ(\sqrt{73})$&$7,11,31,43,47,59,83,103,\ldots$&$19,23,67,71,79,127,\ldots$ \\
\hline
$77c1$&$\BQ(\sqrt{-7})$&$\BQ(\sqrt{11})$&$3,13,17,31,41,47,59,61,\ldots$&$19,83,131,139,167,227,\ldots$ \\
\hline
$80b1$&$\BQ(\sqrt{-1})$&$\BQ(\sqrt{5})$&$3,7,23,43,47,67,83,103,\ldots$&$31,71,79,151,191,199,\ldots$ \\
\hline
$84a1$&$\BQ(\sqrt{-3})$&$\BQ(\sqrt{7})$&$5,11,17,23,41,71,89,101,\ldots$&$47,59,83,131,167,227,\ldots$ \\
\hline
$84b1$&$\BQ(\sqrt{-3})$&$\BQ(\sqrt{7})$&$5,11,17,23,41,71,89,101,\ldots$&$47,59,83,131,167,227,\ldots$ \\
\hline
$89b1$&$\BQ(\sqrt{-1})$&$\BQ(\sqrt{89})$&$3,7,19,23,31,43,59,83,\ldots$&$11,47,67,71,79,107,\ldots$ \\
\hline
$94a1$&$\BQ(\sqrt{-47})$&$\BQ(\sqrt{2})$&$5,11,13,19,29,43,67,107,\ldots$&$23,31,127,151,167,199,\ldots$ \\
\end{supertabular}
\end{center}

\subsection{Examples of the full Birch and Swinnerton-Dyer conjecture}

Let $A$ be the elliptic curve ``$69a1$" with the minimal Weierstrass equation given by
$$
A:  y^2+xy+y=x^3-x-1.
$$
We have $a_2=1$, $a_3=1$ and $a_{23}=-1$. Moreover, $A(\BQ)= \BZ/2\BZ$ and $L^{(alg)}(A,1)= 1/2$. The discriminant of $A$ is $-3^{2} \cdot 23$. The Tamagawa factors $c_3=2$, $c_{23}=1$. Also, a simple computation shows that $\BQ(A[2])=\BQ(\sqrt{-23})$ and $\BQ(A'[2])=\BQ(\sqrt{3})$. Here is the beginning of an infinite set of primes $q$ with good ordinary reduction which are inert in both the fields $\BQ(\sqrt{-23})$ and $\BQ(\sqrt{3})$:
$$
\mathcal{S}=\{5,7,17,19,43,53,67,79,89,103,113,137,149,199, \ldots \}.
$$
Let $M=q_1^*\cdots q_r^*$ be a product of $r$ distinct primes in $\mathcal{S}$. By Theorem \ref{2BSD}, we have 
$$
L(A^{(M)},1) \neq 0,
$$ 
and
$$
\ord_2(L^{\alg}(A^{(M)},1))=r-1.
$$ 
If we carry out a classical $2$-descent on $A^{(M)}$, one shows easily that the $2$-primary component of $\Sha(A^{(M)}/\BQ)$ is zero and $\ord_2(c_{q_i})=1$ for $1 \leq i \leq r$, and therefore the $2$-part of the Birch and Swinnerton-Dyer conjecture holds for $A^{(M)}$. Alternatively, we can just apply Theorem \ref{main2}, take $M \equiv 1 \mod 8$, then the assumption that $2$, $3$ and $23$ all split in $\BQ(\sqrt{M})$ will hold, whence we can also verify the $2$-part of the Birch and Swinnerton-Dyer conjecture. For the full Birch and Swinnerton-Dyer conjecture, in order to apply Theorem 9.3 in Wan's celebrated paper \cite{Wannt}, we need to check the conditions in the theorem. To verify the third one, since $A$ has non-split multiplicative reduction at $23$, we could consider $A^{(5)}$, which has split multiplicative reduction at $23$, and the Tamagawa number is $1$ at $23$, hence the $A[p]|_{G_q}$ ($q=3 \text{ or } 23$) is a ramified representation for any odd prime $p$. Other conditions are easy to verify, so the full Birch and Swinnerton-Dyer conjecture is valid for $A^{(M)}$. Hence the full Birch and Swinnerton-Dyer conjecture is verified for infinitely many elliptic curves. The full Birch and Swinnerton-Dyer conjecture of rank $1$ twists are also accessible in the future work. More examples are given in Wan's paper.

\subsection*{Acknowledgements} We would like to thank John Coates for encouragement, useful discussions and polishings on the manuscript, thank Ye Tian and Xin Wan for helpful advice and comments, and thank Yongxiong Li for helpful comments and carefully reading the manuscript. We  also thank the referee for helpful advice.

\bibliographystyle{alpha}

\end{document}